\documentclass[12pt,reqno,a4paper]{amsart}

\title[On the eigenvalues of a Robin problem] {On the asymptotic
  behaviour of the eigenvalues of a Robin problem}

\author{Daniel Daners}
\author{James B.~Kennedy}

\dedicatory{\upshape
School of Mathematics and Statistics\\
University of Sydney, NSW 2006, Australia\\[.5em]
\texttt{D.Daners@maths.usyd.edu.au\\J.Kennedy@maths.usyd.edu.au}\\[1em]
}

\newtheorem{theorem}{Theorem}[section]
\newtheorem{lemma}[theorem]{Lemma}
\newtheorem{proposition}[theorem]{Proposition}

\theoremstyle{remark}
\newtheorem{remark}[theorem]{Remark}

\numberwithin{equation}{section}

\newcommand{\R}{\mathbb{R}}
\newcommand{\N}{\mathbb{N}}
\DeclareMathOperator{\divergence}{div}
\newcommand*{\subsubset}{\subset\joinrel\subset}

\begin{document}

\begin{abstract}
  We prove that every eigenvalue of a Robin problem with boundary
  parameter $\alpha$ on a sufficiently smooth domain behaves
  asymptotically like $-\alpha^2$ as $\alpha \to \infty$. This
  generalises an existing result for the first eigenvalue.
\end{abstract}



\maketitle


\section{Introduction and Main Results}
\label{sec:intro}

We are interested in the eigenvalue problem
\begin{equation}
    \label{eq:robin}
    \begin{aligned}
        -\Delta u&= \lambda u &\quad &\text{in $\Omega$},\\
        \frac{\partial u}{\partial\nu}& =\alpha u &&
        \text{on $\partial\Omega$}
    \end{aligned}
\end{equation}
where we assume $\Omega\subset\R^N$ is a bounded domain, that is, a
bounded open set, without loss of generality connected, and $\alpha >
0$. The problem \eqref{eq:robin} is usually referred to as a Robin
problem (in comparison with the case $\alpha<0$) or sometimes as a
generalised Neumann problem. This problem has received considerable
attention in the last few years; see for example
\cite{afrouzi:99:pei,daners:09:ip,giorgi:07:ee,giorgi:08:bm,lacey:98:rde,levitin:08:rl,lou:04:spe}
and the references therein.  It is well-known that if $\Omega$ is
Lipschitz then there is a sequence of eigenvalues $\lambda_1 <
\lambda_2 \leq \ldots \to \infty$, which we repeat according to their
multiplicities, where $\lambda_1 < 0$ is simple and is the unique
eigenvalue with a positive eigenfunction $\psi_1$. Our main result is
as follows.

\begin{theorem}
  \label{th:asymp}
  Suppose $\Omega \subset \R^N$ is a bounded domain of class $C^1$.
  Then for every $n \geq 1$ we have
  \begin{equation}
    \label{eq:asymp}
    \lim_{\alpha \to \infty}\frac{\lambda_n(\alpha)}{-\alpha^2} = 1.
  \end{equation}
\end{theorem}

It was shown in \cite{lacey:98:rde} that for $\Omega$ piecewise-$C^1$
the first eigenvalue $\lambda_1$ has the asymptotic behaviour
$\liminf_{\alpha \to \infty} -\lambda_1 (\alpha)/\alpha^2 \geq 1$,
with equality if $\partial \Omega$ is equivalent in some sense to a
sphere. It was also observed in \cite{lacey:98:rde} that when $\Omega$
is a ball of radius $1$, there are $\lfloor \alpha \rfloor + 1$
negative eigenvalues of \eqref{eq:robin}, and they satisfy
$\sqrt{-\lambda_n(\alpha)} \sim \alpha + O(1)$ as $\alpha \to \infty$.
It was subsequently shown in \cite{lou:04:spe} that in fact
\begin{equation}
  \label{eq:lim1}
  \lim_{\alpha \to \infty} \frac{\lambda_1(\alpha)}{-\alpha^2} = 1
\end{equation}
for every bounded and $C^1$ domain $\Omega$.  Related results have
been obtained in \cite{giorgi:07:ee,giorgi:08:bm}. The $C^1$
assumption in \eqref{eq:lim1} is optimal: the authors in
\cite{lacey:98:rde} constructed examples of domains with ``corners''
for which the limit in \eqref{eq:lim1} is a constant larger than one.
Such results were generalised and further studied in
\cite{levitin:08:rl}.

\begin{remark}
  \label{rem:weights}
  One can also consider the same problem with the boundary condition
  $\frac{\partial u}{\partial \nu} = \alpha b u$, where $b \in
  C(\partial \Omega)$ is a weight function which is positive
  somewhere. In this case, if $\Omega$ is bounded and $C^1$, then
  \begin{displaymath}
    \lim_{\alpha \to \infty} \frac{\lambda_1 (\alpha)}{-\alpha^2
    (\max_{\partial \Omega}b)^2} = 1
  \end{displaymath}
  (see \cite[Remark~1.1]{lou:04:spe}). It seems the same should be
  true for $\lambda_n$, $n \geq 1$. However all we can say at present
  is that Theorem~1.1 together with the monotonic behaviour of
  $\lambda_n$ with respect to changes in $b$ imply that
  \begin{displaymath}
    \limsup_{\alpha \to \infty} \frac{\lambda_n (\alpha)}{-\alpha^2
    (\max_{\partial \Omega}b)^2} \leq 1.
  \end{displaymath}
\end{remark}

We will also prove the following result on the eigenfunctions of
\eqref{eq:robin}.

\begin{proposition}
  \label{prop:tozero}
  Suppose $\Omega \subset \R^N$ is bounded and $C^1$. Fix $2\leq p <
  \infty$ and let $\psi_n$ be any eigenfunction associated with
  $\lambda_n$, normalised so that $\|\psi_n\|_{L^p(\Omega)} = 1$. Then
  \begin{itemize}
  \item[(i)] $\psi_n \to 0$ in $L_{loc}^p(\Omega)$ as
    $\alpha \to \infty$;
  \item[(ii)] $\| \psi_n \|_{L^q (\Omega)} \to 0$ as $\alpha \to
    \infty$ for $1 \leq q < p$;
  \item[(iii)] $\| \psi_n \|_{L^r (\Omega)} \to \infty$ as $\alpha \to
    \infty$ for $r > p$.
  \end{itemize}
\end{proposition}

We will prove Theorem~\ref{th:asymp} in the next section and defer the
proof of Proposition~\ref{prop:tozero} until Section~\ref{sec:prop}.
We will use the result of Theorem~\ref{th:asymp} to obtain
Proposition~\ref{prop:tozero}; however, the former is only needed to
show that $\lambda_n (\alpha) \to -\infty$ as $\alpha \to \infty$.
Proposition~\ref{prop:tozero} is valid for Lipschitz domains whenever
we have this more general asymptotic behaviour.

\section{Proof of Theorem~\ref{th:asymp}}
\label{sec:th}

We first discuss the theory related to \eqref{eq:robin} that will be
needed to prove Theorem~\ref{th:asymp}. The form associated with
\eqref{eq:robin} is given by
\begin{displaymath}
  a(u,v) = \int_\Omega \nabla u\cdot \nabla v\,dx-\int_{\partial \Omega}
  \alpha uv\,dx,
\end{displaymath}
where $u, v \in H^1(\Omega)$. We understand eigenvalues $\lambda$ and
associated eigenfunctions $\psi$ of \eqref{eq:robin} in the weak
sense, as satisfying $a(\psi,v) = \lambda \langle \psi, v \rangle$ for
all $v \in H^1(\Omega)$. Here and throughout $\langle
\,.\,,\,.\,\rangle$ denotes the usual inner product on $L^2(\Omega)$.
The eigenfunctions $\psi_1, \psi_2, \ldots$ can be chosen orthogonal
in $L^2(\Omega)$. To see this, note first that if $\lambda_i \neq
\lambda_j$ for some $i,j \geq 1$, then $a (\psi_i, \psi_j) = \lambda_i
\langle \psi_i, \psi_j \rangle = \lambda_j \langle \psi_i, \psi_j
\rangle$ implies $\langle \psi_i, \psi_j \rangle = 0$. If instead
$\lambda_n$ is a repeated eigenvalue, we may apply the Gram-Schmidt
process to its eigenfunctions. We also impose the scaling $\|\psi_n
\|_{L^2(\Omega)} = 1$ in this section.  With the eigenvalues ordered
by increasing size and repeated according to their multiplicities, the
$n$th eigenvalue may be characterised variationally as
\begin{equation}
  \label{eq:minimax}
  \lambda_n (\alpha) = \inf_{0 \neq v \in M_n} 
  \frac{a(v,v)}{\|v\|_{L^2(\Omega)}^2},
\end{equation}
where $M_n$ is the subspace of $H^1(\Omega)$ of codimension $n-1$
obtained by taking the orthogonal complement of the $L^2$-span of the
first $n-1$ eigenfunctions $\psi_1, \ldots, \psi_{n-1}$ (see
\cite[Section~VI.1]{courant:53:mmp}). If we set $v_n := v -
\sum_{i=1}^{n-1} \langle v, \psi_i \rangle \psi_i$, then $v_n \in M_n$
and so provided $v_n \neq 0$, that is, provided $v$ is not in the
$L^2$-span of $\psi_1,\ldots,\psi_{n-1}$, we may use $v_n$ as a test
function in \eqref{eq:minimax} to estimate $\lambda_n$ from above.

We will use this representation, together with an appropriate choice
of $v$ and an induction argument on $n$, to prove
Theorem~\ref{th:asymp}. Our choice of test function is due to an
argument in \cite[Theorem~2.3]{giorgi:07:ee}, though also
cf.~\cite[Example~2.4]{levitin:08:rl}. We will assume throughout that
$\Omega \subset \R^N$ is bounded and $C^1$, although some of the
results, including the next lemma, are valid for Lipschitz domains
with the same proof.

\begin{lemma}
  \label{lemma:test}
  Let $d\in \R^N$, $\|d\| = 1$ be any unit vector. Set $u_d(x,
  \alpha):= c e^{\alpha x \cdot d} \in C^\infty(\R^N) \cap
  H^1(\Omega)$, where $c = c(d, \alpha)$ is a constant chosen so that
  $\|u_d\|_{L^2(\Omega)} = 1$. Then $a(u_d, u_d) \leq -\alpha^2$ for
  all $\alpha > 0$.
\end{lemma}

\begin{proof}
  For $x \in \R^N$ writing $x = (x_1, \ldots, x_N)$, we may without
  loss of generality rotate our coordinate system if necessary so that
  $d = (0,\ldots,0,1)$. In this case $u_d = ce^{\alpha x_N}$ and
  $\nabla u_d = (0, \ldots, 0, c\alpha e^{\alpha x_N})$. Hence
  \begin{displaymath}
    a(u_d,u_d) = c^2 \alpha^2\int_\Omega e^{2\alpha x_N}\,dx - c^2
    \alpha \int_{\partial \Omega} e^{2\alpha x_N}\,d\sigma.
  \end{displaymath}
  We will now use the divergence theorem on $V:= (0, \ldots, 0,
  e^{2\alpha x_N}) \in C^\infty (\R^N, \R^N)$ and the domain $\Omega$
  (see for example \cite[Th\'eor\`eme~3.1.1]{necas:67:ee}). Denoting
  the outer unit normal to $\Omega$ by $\nu_\Omega (x) = (\nu_1(x),
  \ldots, \nu_N(x))$, $x \in \partial \Omega$, we have
  \begin{displaymath}
    \begin{split}
      \int_{\partial\Omega} e^{2\alpha x_N}\,d\sigma 
      &\geq \int_{\partial\Omega} e^{2\alpha x_N} 
      \nu_N\,d\sigma = \int_{\partial\Omega} V\cdot\nu_\Omega\,d\sigma\\
      &=\int_\Omega \divergence V\,dx
      = 2\alpha \int_\Omega e^{2\alpha x_N}\,dx.
    \end{split}
  \end{displaymath}
  Multiplying through by $\alpha>0$ and combining this with the
  expression for $a(u_d,u_d)$ yields
  \begin{displaymath}
    a(u_d,u_d) \leq -\alpha^2 c^2 \int_\Omega e^{2\alpha x_N}\,dx
    = -\alpha^2,
  \end{displaymath}
  where the last equality follows from the definition of $c$.
\end{proof}

\begin{remark}
  \label{rem:halfspace}
  An easy calculation shows that the function $u(x):= e^{\alpha x_N}$
  is a positive eigenfunction, with eigenvalue $-\alpha^2$, of
  \eqref{eq:robin} on the half-space $T = \{x \in \R^N: x_N < 0 \}$.
\end{remark}

For $d \in \R^N$ a fixed unit vector and $n \geq 1$ also fixed, set
$u_{n+1}:= u_d - \sum_{i=1}^{n} \langle u_d, \psi_i \rangle \psi_i \in
M_{n+1}$. We will use $u_{n+1}$ as a test function in the variational
characterisation in order to establish \eqref{eq:asymp}. To that end,
we estimate $\lambda_{n+1}$ in terms of the previous $n$ eigenvalues
and functions.

\begin{lemma}
  \label{lemma:ind}
  Suppose $u_d \not\in \textrm{span}\{\psi_1, \ldots, \psi_n\}$. Then
  \begin{equation}
    \label{eq:ind}
    \lambda_{n+1}(\alpha) \leq \frac{-\alpha^2
      -\sum_{i=1}^{n}\lambda_i \langle u_d,\psi_i \rangle^2}
    {1 - \sum_{i=1}^{n} \langle u_d,\psi_i \rangle^2}.
  \end{equation}
\end{lemma}

\begin{proof}
  Since $u_d$ is not a linear combination of the first $n$
  eigenfunctions, we can use $u_{n+1} = u_d - \sum_{i=1}^{n} \langle
  u_d, \psi_i \rangle {\psi_i} \not\equiv 0$ as a test function in
  \eqref{eq:minimax}. A simple calculation using the orthonormality
  of the eigenfunctions shows that
  \begin{displaymath}
    0<\langle u_{n+1},u_{n+1}\rangle
    =1-\sum_{i=1}^{n}\langle u_d,\psi_i\rangle^2.
  \end{displaymath}
  We now estimate $a(u_{n+1},u_{n+1})$. Using the definition of
  $u_{n+1}$ and the bilinearity of the form $a$, we see that
  $a(u_{n+1},u_{n+1})$ is given by
  \begin{displaymath}
    a(u_d,u_d)-2\sum_{i=1}^n \langle u_d,\psi_i
    \rangle\, a(u_d,\psi_i)\\ + \sum_{i=1}^n \sum_{j=1}^n \langle u_d,
    \psi_i \rangle^2 a(\psi_i,\psi_j).
  \end{displaymath}
  Since $a(u_d,\psi_i) = \lambda_i \langle u_d,\psi_i \rangle$, and
  since $a(\psi_i, \psi_j) = \lambda_i$ if $i=j$ and $0$ otherwise, we
  obtain
  \begin{displaymath}
    a(u_{n+1},u_{n+1}) = a(u_d,u_d) - \sum_{i=1}^n \lambda_i 
    \langle u_d, \psi_i \rangle^2.
  \end{displaymath}
  (Cf.~the abstract theory in \cite[Section~I.6.10]{kato:76:pt}.)
  Using the estimate of $a(u_d,u_d)$ from Lemma~\ref{lemma:test} and
  putting everything together yields
  \begin{displaymath}
    \lambda_{n+1}(\alpha) \leq \frac{a(u_{n+1},u_{n+1})}{\|u_{n+1}
      \|_{L^2(\Omega)}^2} \leq \frac{-\alpha^2
      -\sum_{i=1}^{n}\lambda_i \langle u_d,\psi_i \rangle^2}
    {1 - \sum_{i=1}^{n} \langle u_d,\psi_i \rangle^2},
  \end{displaymath}
  establishing \eqref{eq:ind}.
\end{proof}

Roughly speaking, to prove Theorem~\ref{th:asymp} using the estimate
of $\lambda_{n+1}$ in Lemma~\ref{lemma:ind} we have to prove that we
can find a direction $d$ such that $\langle u_d, \psi_i \rangle$ stays
small as $\alpha \to \infty$ for all $1 \leq i \leq n$.  To that end
we will study the functions $u_d$ more carefully. We start by
observing that, for any given $\alpha>0$, the upper level sets of
$u_d$ are restrictions to $\Omega$ of half-planes of the form $\{x \in
\R^N: x \cdot d > \kappa \}$, where $\kappa \in \R$. The key place
where we will use the assumption that $\Omega$ has $C^1$ boundary is
in parts (iii) and (iv) of the next lemma.

\begin{lemma}
  \label{lemma:rot}
  Let $d \in \R^N$, $\|d\|=1$. For $\kappa \in \R$ set
  \begin{equation}
    \label{eq:usets}
    \begin{aligned}
      {U_d}(\kappa)&:= \{x \in \Omega: x \cdot d > \kappa\},\\
      {\kappa_d}&:= \sup \{\kappa \in \R: U_d (\kappa) \neq \emptyset\},\\
      {K_d}&:= \{x \in \overline \Omega: x \cdot d = \kappa_d \}.
    \end{aligned}
  \end{equation}
  Then
  \begin{itemize}
  \item[(i)] the $U_d (\kappa)$ are open, nested (i.e.~$U_d(\kappa)
    \subset U_d (\kappa')$ if $\kappa > \kappa'$), nonempty if and
    only if $\kappa< \kappa_d$, and $0 \neq |U_d (\kappa)| \to 0$ as
    $\kappa \to \kappa_d$ from below;
  \item[(ii)] $\emptyset \neq K_d \subset \partial \Omega$;
  \item[(iii)] if $z \in K_d$, then $d = \nu_\Omega(z)$, the outer
    unit normal to $\Omega$ at $z$;
  \item[(iv)] if $d \neq e \in \R^N$, $\|e\|=1$ is another unit vector
    with $U_e(\kappa)$ and $\kappa_e$ defined as in \eqref{eq:usets},
    then there exists $\varepsilon > 0$ such that $U_d (\kappa) \cap
    U_e (\tilde \kappa) = \emptyset$ for all $\kappa \in (\kappa_d -
    \varepsilon, \kappa_d)$ and all $\tilde\kappa \in (\kappa_e -
    \varepsilon, \kappa_e)$.
  \end{itemize}
\end{lemma}

\begin{proof}
  (i) is obvious. For (ii), to show $K_d \neq \emptyset$ we note that
  $K_d = \cap_{\kappa < \kappa_d} \overline{U_d(\kappa)}$, that is,
  $K_d$ is the intersection of nested, compact and nonempty sets. That
  $K_d \subset \partial \Omega$ is immediate from the definitions and
  the fact that the $U_d$ are open. For (iii), we assume as in the
  proof of Lemma~\ref{lemma:test} that $d = (0,\ldots,0,1)$, so that
  $U_d(\kappa) = \{ x \in \Omega: x_N > \kappa \}$. Then $z =
  (z_1,\ldots, z_N) \in K_d$ means $z_N = \kappa_d$, that is, $z_N =
  \max \{x_N: x \in \overline \Omega \}$. Since $\Omega$ is $C^1$,
  this means the tangent plane to $\Omega$ at $z \in K_d$ must be
  horizontal. Thus $\nu_\Omega(z)$ points in the direction $x_N$, that
  is, $\nu_\Omega(z) = (0,\ldots,0,1)$.  For (iv), suppose for a
  contradiction that there exist $\kappa_j \nearrow \kappa_d$ and
  $\tilde\kappa_j \nearrow \kappa_e$ such that, for each $j \geq 1$,
  there exists $x_j \in U_d (\kappa_j) \cap U_e (\tilde\kappa_j)$.
  Since $\overline \Omega$ is compact, a subsequence of the $x_j$
  converges to some $z \in \overline \Omega$.  Since $x_j \in
  U_d(\kappa_j)$ and $\cap_{j \geq 1} \overline{U_d(\kappa_j)} = K_d$,
  we see $z \in K_d$. By a similar argument, $z \in K_e$.  This
  contradicts (iii) since $d \neq e$.
\end{proof}

We now show that for $d$ fixed, all the mass of $u_d$ becomes
concentrated in an arbitrarily small region of $\Omega$ as $\alpha \to
\infty$.

\begin{lemma}
  \label{lemma:u}
  Let $d \in \R^N$ and $u_d (x) = c e^{\alpha x \cdot d}$ be as in
  Lemma~\ref{lemma:test} and let $U_d(\kappa)$ and $\kappa_d$ be as in
  Lemma~\ref{lemma:rot}. For every $\varepsilon > 0$ and $\kappa' <
  \kappa_d$ there exists $\alpha_\varepsilon := \alpha (\varepsilon,
  \kappa') > 0$ such that
  \begin{equation}
    \label{eq:ubound}
    \|u_d\|_{L^2 (\Omega \setminus U_d(\kappa'))}^2 < \varepsilon
  \end{equation}
  for all $\alpha > \alpha_\varepsilon$.
\end{lemma}

\begin{proof}
  Since $u_d (x) \leq c e^{\alpha \kappa'}$ for all $x \in \Omega
  \setminus U_d (\kappa')$, we have
  \begin{displaymath}
    \|u_d\|_{L^2 (\Omega \setminus U_d(\kappa'))}^2
    \leq ce^{2\alpha \kappa'}|\Omega|.
  \end{displaymath}
  Choose $\kappa'' \in (\kappa', \kappa_d)$.  Then $U_d (\kappa'')
  \subset U_d (\kappa')$ with $|U_d (\kappa'')| \neq 0$ and
  \begin{displaymath}
    1=\|u_d\|_{L^2(\Omega)}^2\geq\|u_d\|_{L^2(U_d(\kappa''))}^2
    \geq c e^{2\alpha \kappa''}|U_d (\kappa'')|.
  \end{displaymath}
  For $\varepsilon > 0$ fixed, choose $\alpha_\varepsilon > 0$ such
  that
  \begin{equation}
    \label{eq:alphaep}
    e^{2\alpha_\varepsilon \kappa'}|\Omega| < \varepsilon 
    e^{2\alpha_\varepsilon \kappa''}|U_d(\kappa'')|,
  \end{equation}
  which we can do since $\kappa' < \kappa''$.  Then
  \eqref{eq:alphaep} will hold uniformly in $\alpha >
  \alpha_\varepsilon$ and so
  \begin{displaymath}
    \|u_d\|_{L^2 (\Omega \setminus U_d(\kappa'))}^2 <
    c e^{2\alpha \kappa'}|\Omega| < \varepsilon 
    c e^{2\alpha \kappa''}|U_d(\kappa'')| < \varepsilon
  \end{displaymath}
  for all $\alpha > \alpha_\varepsilon$.
\end{proof}

Lemma~\ref{lemma:u} implies that for fixed $d$, $u_d \rightharpoonup
0$ weakly in $L^2 (\Omega)$ as $\alpha \to \infty$; it turns out that
the same is true of the $\psi_i$ (see Proposition~\ref{prop:tozero}).
But this is not enough to show directly that $\langle u_d, \psi_i
\rangle$ is uniformly small, since both $u_d$ and $\psi_i$ vary with
$\alpha$.  Instead, we will use the following rather technical result
concerning the $u_d$. Since this does not use any specific properties
of the $\psi_i$, we set it up so it works for arbitrary
$L^2$-functions.

\begin{lemma}
  \label{lemma:finite}
  Fix $n \geq 1$ and $\delta>0$. Suppose we have a sequence $\alpha_k
  \to \infty$ and for each $k \in \N$ a family of $n$ functions
  $\varphi_i (k) \in L^2(\Omega)$, $1 \leq i \leq n$, such that
  $\|\varphi_i (k)\|_{L^2(\Omega)} = 1$ for all $1 \leq i \leq n$ and
  $k \in \N$.  Then there exists a unit vector $d \in \R^N$ and a
  subsequence $\alpha_{k_l} \to \infty$ of the $(\alpha_k)$ such that
  \begin{equation}
    \label{eq:phibound}
    \sum_{i=1}^n \langle u_d(k_l), \varphi_i(k_l) \rangle^2 \leq \delta,
  \end{equation}
  for all $l \in \N$, where $u_d(k_l) = u_d(x, \alpha_{k_l})$ is as in
  Lemma~\ref{lemma:test}.
\end{lemma}

\begin{proof}
  Fix $n \geq 1$, $\delta>0$ and a sequence $\alpha_k \to \infty$.
  Choose $m \geq 1$ and $\varepsilon > 0$, to be specified precisely
  later on. Now choose any $m$ distinct unit vectors $d_j \in \R^N$,
  $1 \leq j \leq m$, and for each $j$ let $u_j := u_{d_j}(x,
  \alpha_k)$ be as in Lemma~\ref{lemma:test}. For each $j$ choose a
  nonempty open set $U_j := U_{d_j} (\kappa_j)$ as in
  Lemma~\ref{lemma:rot}.  By making an appropriate choice of
  $\kappa_j$ we may assume the $U_j$ are pairwise disjoint. Using
  Lemma~\ref{lemma:u}, we find an $\alpha_\varepsilon > 0$ such that
  \begin{displaymath}
    \|u_j\|_{L^2(\Omega \setminus U_j)}^2 < \varepsilon
  \end{displaymath}
  for all $\alpha > \alpha_\varepsilon$ and all $1 \leq j \leq m$. By
  discarding at most finitely many $k$, we may assume $\alpha_k >
  \alpha_\varepsilon$ for all $k \in \N$. Now for each $k \in \N$, we
  have
  \begin{displaymath}
    \int_\Omega \sum_{i=1}^n |\varphi_i(k)|^2\,dx = \sum_{i=1}^n
    \|\varphi_i(k)\|_{L^2(\Omega)}^2 = n.
  \end{displaymath}
  Since the $U_j$ are disjoint, it follows that for each $k \in \N$,
  there exists at least one $j = j_k$ such that
  \begin{displaymath}
    \int_{U_{j_k}} \sum_{i=1}^n |\varphi_i(k)|^2\,dx \leq \frac{n}{m}.
  \end{displaymath}
  For this $j_k$, using H\"older's inequality, for each $1 \leq i \leq
  n$ we have
  \begin{displaymath}
    \begin{split}
      |\langle u_{j_k},\varphi_i(k)\rangle| &\leq \int_{U_{j_k}} |u_j 
      \varphi_i|\,dx + \int_{\Omega \setminus U_{j_k}} |u_j\varphi_i|\,dx\\
      &\leq \|u_j\|_{L^2(\Omega)}\Bigl(\frac{n}{m}\Bigr)^\frac{1}{2} +
      \varepsilon^\frac{1}{2}\|u_j\|_{L^2(\Omega)}\|\varphi_i\|_{L^2(\Omega)}\\
      &=\Bigl(\frac{n}{m} \Bigr)^\frac{1}{2}+\varepsilon^\frac{1}{2},
    \end{split}
  \end{displaymath}
  where we have used the bound $\int_{U_j} |\varphi_i|^2\,dx
  \leq n/m$. We now specify $m \geq 1$ and
  $\varepsilon > 0$ to be such that
  \begin{displaymath}
    n \Bigl(\Bigl(\frac{n}{m} \Bigr)^\frac{1}{2}+\varepsilon^\frac{1}{2}
    \Bigr)^2 \leq \delta,
  \end{displaymath}
  noting that this depends only on $n$ and $\delta$. Squaring the
  above estimate for $|\langle u_{j_k}, \varphi_i(k) \rangle |$ and
  summing over $i$, this implies that for all but finitely many $k \in
  \N$, \eqref{eq:phibound} holds for at least one of the $m$ fixed
  $u_j$.

  By a simple counting argument, there must exist at least one $j^*$
  between $1$ and $m$ such that \eqref{eq:phibound} holds for this
  fixed $u_{j^*}$ and infinitely many $\alpha_k$. This gives us our
  $u_d$ and $(\alpha_{k_l})$.
\end{proof}

\begin{proof}[Proof of Theorem~\ref{th:asymp}]
  The proof is by induction on $n$. The step when $n=1$ is given by
  \cite[Theorem~1.1]{lou:04:spe}.  Now fix $n \geq 1$ and suppose we
  know that for all $1 \leq i \leq n$, $-\lambda_i(\alpha_k)
  /\alpha_k^2 \to 1$ as $k \to \infty$ for every sequence $\alpha_k
  \to \infty$. It suffices to prove that for every such sequence
  $\alpha_k \to \infty$, there exists a subsequence $\alpha_{k_l} \to
  \infty$ such that $-\lambda_{n+1}(\alpha_{k_l})/\alpha_{k_l}^2 \to
  1$ as $l \to \infty$.

  So fix a particular sequence $\alpha_k \to \infty$ and also fix $0 <
  \delta < 1$.  Let $u_d$ satisfy the conclusion of
  Lemma~\ref{lemma:finite} for a subsequence which we will still
  denote by $(\alpha_k)$, this $\delta > 0$ and the family of $n$
  functions $\psi_i (\alpha_k) =: \varphi_i(k)$, $1 \leq i \leq n$.
  Then by Lemma~\ref{lemma:finite} we know that
  \begin{equation}
    \label{eq:psibound}
    \sum_{i=1}^n \langle u_d(\alpha_k), \psi_i(\alpha_k)\rangle^2\leq\delta
  \end{equation}
  for all $k \in \N$ and the fixed direction $d$. In particular,
  \eqref{eq:psibound} implies $u_d \not\in \textrm{span}
  \{\psi_1(\alpha_k), \ldots, \psi_n (\alpha_k) \}$ for any $k \in
  \N$, since $\delta < 1$.  Applying Lemma~\ref{lemma:ind} to $u_d$
  for each $k \in \N$, we obtain
  \begin{displaymath}
    \lambda_{n+1}(\alpha_k)\leq\frac{-\alpha_k^2
      - \sum_{i=1}^n \lambda_i \langle u_d, \psi_i\rangle^2}
    {1 - \sum_{i=1}^n \langle u_d, \psi_i\rangle^2}
  \end{displaymath}
  for all $k \in \N$. This implies
  \begin{equation}
    \label{eq:subseqbound}
    \frac{\lambda_1(\alpha_k)}{-\alpha_k^2}\geq
    \frac{\lambda_{n+1}(\alpha_k)}{-\alpha_k^2}\geq
    \frac{1-\sum_{i=1}^n \frac{\lambda_i(\alpha_k)}{-\alpha_k^2}
      \,\langle u_d,\psi_i\rangle^2}{1-\sum_{i=1}^n \langle u_d,\psi_i
      \rangle^2}.
  \end{equation}
  Using the bound \eqref{eq:psibound}, which holds independently of $k
  \in \N$, together with the induction assumption
  $-\lambda_i(\alpha_k^2) /\alpha_k^2 \to 1$ as $k \to \infty$ for
  all $i \leq n$ it follows that the term on the right in
  \eqref{eq:subseqbound} converges to $1$ as $k \to \infty$. This
  establishes the desired limit for $-\lambda_{n+1}(\alpha_k)/
  \alpha_k^2$, which completes the proof.
\end{proof}

\section{Proof of Proposition~\ref{prop:tozero}}
\label{sec:prop}

Fix $n \geq 1$ and $p \geq 2$. We first obtain the following interior
estimate for $\psi_n$, from which the proof of the proposition will
follow easily.

\begin{lemma}
  \label{lemma:int}
  Under the assumptions of Proposition~\ref{prop:tozero}, if $\varphi
  \in C_c^\infty (\Omega)$, then
    \begin{displaymath}
      \lambda_n \geq -(p-1)^{-1}\frac{\int_{\Omega}|\psi_n|^p \,|\nabla 
        \varphi|^2 \,dx}{\int_{\Omega} |\psi_n|^p \,\varphi^2\,dx}
    \end{displaymath}
    for all $\alpha > 0$ and all $n \geq 1$.
\end{lemma}

\begin{proof}
  Given $\varphi \in C_c^\infty (\Omega)$, we will use $\phi:=
  \varphi^2 |\psi_n|^{p-2} \psi_n$ as a test function in the weak form
  of \eqref{eq:robin} given by
  \begin{equation}
    \label{eq:weak}
    \lambda_n \int_\Omega \psi_n v\,dx = a(\psi_n, v)
    =\int_\Omega \nabla \psi_n\cdot\nabla v\,dx-\int_{\partial
      \Omega} \alpha\psi_n v\,d\sigma
  \end{equation}
  for all $v \in H^1 (\Omega)$. We first note that if $p \geq 2$, then
  since $\psi_n \in C(\overline \Omega)$ (see
  \cite[Corollary~4.2]{daners:09:ip}) we have $\phi \in H^1 (\Omega)$
  with $\nabla \phi = 2\varphi|\psi_n|^{p-2} \psi_n \nabla \varphi +
  (p-1) \varphi^2 |\psi_n|^{p-2} \nabla \psi_n$. Moreover $\langle
  \phi, \psi_n \rangle = \int_\Omega \varphi^2 |\psi_n|^p\,dx \neq 0$,
  since $\psi_n$ cannot vanish identically on an open set (see
  \cite{aronszajn:57:uce}).  Hence, by completing the square,
  \begin{displaymath}
    \begin{split}
      \int_\Omega \nabla &\psi_n \cdot \nabla \phi\,dx\\
      &= \int_\Omega 2\varphi|\psi_n|^{p-2}\psi_n \nabla \varphi \cdot
      \nabla \psi_n + (p-1)\varphi^2 |\psi_n|^{p-2}|\nabla \psi_n|^2\,dx\\
      &= \int_\Omega \Bigl|(p-1)^\frac{1}{2}|\psi_n|^{\frac{p}{2}-1}\varphi
      \nabla \psi_n + (p-1)^{-\frac{1}{2}}|\psi_n|^{\frac{p}{2}-1}\psi_n
      \nabla \varphi \Bigr|^2\,dx\\
      &\qquad\qquad -\int_\Omega (p-1)^{-1}|\psi_n|^p 
      |\nabla \varphi|^2\,dx.
    \end{split}
  \end{displaymath}
  Substituting this into \eqref{eq:weak}, and using that $\varphi
  \equiv 0$ on $\partial \Omega$,
  \begin{displaymath}
    \lambda_n \int_\Omega \varphi^2 |\psi_n|^p\,dx
    = \int_\Omega \nabla \psi_n \cdot \nabla \phi\,dx
    \geq -\int_\Omega (p-1)^{-1}|\psi_n|^p|\nabla\varphi|^2\,dx.
  \end{displaymath}
  Rearranging gives the conclusion of the lemma.
\end{proof}

To prove the proposition, part (i) uses the result of
Theorem~\ref{th:asymp}, that $\lambda_n \to -\infty$ as $\alpha \to
\infty$; parts (ii) and (iii) follow directly from (i).

\begin{proof}[Proof of Proposition~\ref{prop:tozero}]
  (i) Fix $p \geq 2$, $n \geq 1$ and $\Omega_0 \subsubset \Omega$ and
  assume $\|\psi_n\|_{L^p(\Omega)} = 1$. Let $\varphi \in C_c^\infty
  (\Omega)$ be such that $0 \leq \varphi \leq 1$ in $\Omega$ and
  $\varphi \equiv 1$ in $\Omega_0$. Setting $K:= (p-1)^{-1} \|\nabla
  \varphi\|_{L^\infty (\Omega)}^2 > 0$, which depends only on $p$ and
  $\Omega_0$, by Lemma~\ref{lemma:int},
  \begin{displaymath}
    \lambda_n \geq \frac{-K}{\int_{\Omega_0} |\psi_n|^p\,dx}
  \end{displaymath}
  for all $\alpha > 0$. Since $\lambda_n \to -\infty$ as $\alpha \to
  \infty$ by Theorem~\ref{th:asymp}, this forces $\int_{\Omega_0}
  |\psi_n|^p\,dx \to 0$ as $\alpha \to \infty$.

  (ii) Fix $1 \leq q < p$ and $\varepsilon > 0$. Choose
  $\Omega_\varepsilon \subsubset \Omega$ such that $|\Omega \setminus
  \Omega_\varepsilon|^\frac{p-q}{p} < \varepsilon/2$, which we may do
  since $p > q$. Also choose $\alpha_\varepsilon > 0$ such that
  \begin{displaymath}
    \| \psi_n \|_{L^p (\Omega_\varepsilon)}^q < \frac{\varepsilon}{2}
    |\Omega_\varepsilon|^\frac{q-p}{p}
  \end{displaymath}
  for all $\alpha > \alpha_\varepsilon$, which we may do by (i).
  Noting that $p/q$ and $p/(p-q)$ are dual exponents, H\"older's
  inequality implies
  \begin{displaymath}
    \begin{split}
      \|\psi_n\|_{L^q(\Omega)}^q
      &= \int_{\Omega_\varepsilon} |\psi_n|^q\,dx + \int_{\Omega
        \setminus \Omega_\varepsilon} |\psi_n|^q\,dx\\
      &\leq \Bigl(\int_{\Omega_\varepsilon} |\psi_n|^p \,dx\Bigr)^\frac{q}{p}
      |\Omega_\varepsilon|^\frac{p-q}{p} + \Bigl(\int_{\Omega
        \setminus \Omega_\varepsilon}|\psi_n|^p \,dx\Bigr)^\frac{q}{p}
      |\Omega \setminus \Omega_\varepsilon|^\frac{p-q}{p}\\
      &= \|\psi_n\|_{L^p(\Omega_\varepsilon)}^q |\Omega_\varepsilon
      |^\frac{p-q}{p} + \|\psi_n\|_{L^p(\Omega \setminus 
        \Omega_\varepsilon)}^q |\Omega \setminus 
      \Omega_\varepsilon|^\frac{p-q}{p} < \varepsilon
    \end{split}
  \end{displaymath}
  for all $\alpha > \alpha_\varepsilon$, by choice of
  $\Omega_\varepsilon$ and $\alpha_\varepsilon$, and since
  $\|\psi_n\|_{L^p(\Omega \setminus \Omega_\varepsilon)}^q \leq 1$.

  (iii) Fix $r > p$. If we normalise $\psi_n$ so that $\|\psi_n
  \|_{L^r (\Omega)} = 1$, then (ii) implies $\|\psi_n\|_{L^p (\Omega)}
  \to 0$, so that
  \begin{equation}
    \label{eq:ratio}
    \frac{\|\psi_n\|_{L^r(\Omega)}}{\|\psi_n\|_{L^p(\Omega)} }
    \longrightarrow \infty
  \end{equation}
  as $\alpha \to \infty$. Now re-normalise so that
  $\|\psi_n\|_{L^p(\Omega)} = 1$. Since this does not affect
  \eqref{eq:ratio}, in this case $\|\psi_n\|_{L^r(\Omega)} \to
  \infty$.
\end{proof}

\bibliographystyle{amsplain}

\providecommand{\bysame}{\leavevmode\hbox to3em{\hrulefill}\thinspace}
\providecommand{\href}[2]{#2}

\end{document}